\documentclass{amsart}
\usepackage{latexsym, amssymb, enumerate, amsmath}
\usepackage{epsfig}
\usepackage{amstext}
\usepackage{amsgen}
\usepackage{amsxtra}
\usepackage{amsgen}
\usepackage{enumitem}
\usepackage{subfigure}

\numberwithin{equation}{section}

\newcommand{\R}{\mathbb{R}}

\newcommand{\sgn}{\text{sign}}

\newcommand{\prefac}{\frac{1}{\varepsilon}}
\newcommand{\prefactwo}{\frac{1}{2\varepsilon}}

\newtheorem{thm}{Theorem}[section]
\newtheorem{prop}[thm]{Proposition}
\newtheorem{lem}[thm]{Lemma}

\newtheorem{rem}[thm]{Remark}

\begin{document}

\title[Asymptotic behavior of a price formation model]{On the asymptotic behavior of a Boltzmann-type price formation model}
\author[M. Burger]{Martin Burger$^1$}
\address{$^1$Institute for Computational and Applied Mathematics, University of M\"unster, Einsteinstrasse 62, 48149 M\"unster, Germany}

\author[L. Caffarelli]{Luis Caffarelli$^2$}
\address{$^2$University of Texas at Austin, 1 University Station, C120, Austin, Texas 78712-1082, USA }

\author[P.A. Markowich]{Peter A. Markowich$^3$}
\address{$^3$ 4700 King Abdullah University of Science and Technology, Thuwal 23955-6900, Kingdom of Saudi Arabia}

\author[M.-T. Wolfram]{Marie-Therese Wolfram$^3$}

\maketitle

\begin{abstract}
In this paper we study the asymptotic behavior of a Boltzmann type price formation
model, which describes the trading dynamics in a financial market. In many of these markets
trading happens at high frequencies and low transactions costs. This observation motivates the study
of the limit as the number of transactions $k$ tends to infinity, the transaction cost $a$ to zero and
$ka=const$. Furthermore we illustrate the
price dynamics with numerical simulations.
\end{abstract}

\section{Introduction}
\noindent According to O'Hara \cite{OHara1998} financial markets are characterized by two functions: first by providing
liquidity and second by facilitating the price. The evolution of the price emerges from the microscopic
trading strategies of the players and the trading system considered. High frequency trading (HFT) is an
automated trading strategy, which is carried out by computers that place and withdraw orders within
milli- or even microseconds. In 2012 HFT accounted for approximately $52\%$ of the overall US equity trading volume.

\noindent This note focuses on the asymptotic behavior of markets, where the price dynamics of a traded good
are determined by the following situation: Consider a large number of vendors and a large number of buyers trading a specific good. If a buyer and a vendor agree on a price $p = p(t)$ a transaction takes place. 
The price of this transaction is given by a positive constant $a \in \R^+$.
After the transaction, the buyer and vendor immediately switch places. Since the actual cost for the buyer is $p(t)+a$,
he/she will sell the good for at least that price. The profit for the vendor is $p(t)-a$, hence he/she will 
try to buy the good for a price lower than $p(t)-a$. \\

\noindent Based on the situation described above Lasry \& Lions \cite{LL2007} proposed the following parabolic free boundary price formation model:
\begin{subequations}\label{e:lasrylions}
\begin{align}
f_t(x,t) &= \frac{\sigma^2}{2} f_{xx}(x,t) + \lambda(t) \delta(x-p(t)+a) \text{ for } x < p(t) \text{ and } f(x,t) = 0 \text{ for } x > p(t)\\
g_t(x,t) &= \frac{\sigma^2}{2} g_{xx}(x,t) + \lambda(t) \delta(x-p(t)-a) \text{ for } x > p(t) \text{ and } g(x,t) = 0 \text{ for } x < p(t).
\end{align}
\end{subequations}
The functions $f = f(x,t)$ and $g = g(x,t)$ denote the density of buyers and vendors and $a \in \R^+$ the transaction costs.
The agreed price $p = p(t)$ enters as a free boundary and $\lambda(t) = -f_x(p(t),t) = g_x(p(t),t)$. Trading events take only place at
 the price $p = p(t)$, since the density of buyers and vendors is zero for prices smaller or larger than $p(t)$. The Lasry \& Lions model 
was analyzed in a series of papers, see \cite{MMPW2009, CMP2011, CMW2011, CGGK2009, GG2011}.\\

\noindent Lasry \& Lions motivated their model using mean field game theory, but did not discuss its microscopic origin. 
The lack of understanding system \eqref{e:lasrylions} on the microscopic level motivated further research in this direction. In 
\cite{BCMW2013} we considered a simple agent based model with standard stochastic price fluctuations and discrete trading events. 
This  Boltzmann-type price formation (BPF) model reads as
\begin{subequations}\label{e:boltzprice}
\begin{align}
f_t(x,t) &=  \frac{\sigma^2}{2} f_{xx}(x,t) - k f(x,t) g(x,t) + k f(x+a,t) g(x+a,t) \\
g_t(x,t) &=  \frac{\sigma^2}{2} g_{xx}(x,t) - k f(x,t) g(x,t) + k f(x-a,t) g(x-a,t).
\end{align}
with initial data
\begin{align}
f(x,0) = f_I(x) \geq 0, ~~g(x,0) = g_I(x) \geq 0, 
\end{align}
\end{subequations}
independent of $k$. In system \eqref{e:boltzprice} the parameter $k$ denotes the transaction rate and $\sigma$ the diffusivity. 
The total number of transactions at a price $x$ is given by
\begin{align}\label{e:trade}
\mu(x,t) = k f(x,t) g(x,t).
\end{align}

\noindent One of the fundamental differences between \eqref{e:lasrylions} and \eqref{e:boltzprice} 
is the fact that trading events in the first take only place at the price $p=p(t)$. In BPF \eqref{e:boltzprice} a good 
can be traded at any price, with a rate $\mu$ given by \eqref{e:trade}. Then the mean, median and
maximum of $\mu$ gives an estimate for the price. \\
There is however a strong connection between the BPF model \eqref{e:boltzprice} and \eqref{e:lasrylions}. We showed 
that solutions of \eqref{e:boltzprice} converge 
to solutions of \eqref{e:lasrylions} as the transaction rate $k$ tends to infinity, see \cite{BCMW2013}.
This finding motivated further research on different asymptotic limits, for example by considering high trading 
frequencies and little transaction costs. This market behavior corresponds to the case $k\rightarrow \infty$, 
$a \rightarrow 0$ with $ka = c$. For studying this limit rewrite system \eqref{e:boltzprice} as
\begin{subequations}\label{e:rewrite1}
\begin{align}
f_t(x,t) &= c \frac{(fg)(x+a,t) - (fg)(x,t)}{a} + \frac{\sigma^2}{2}f_{xx}(x,t)\label{e:flimit} \\
g_t(x,t) &= c \frac{(fg)(x-a,t) - (fg)(x,t)}{a} + \frac{\sigma^2}{2} g_{xx}(x,t)\label{f:glimit}.
\end{align}
\end{subequations}
We showed that \eqref{e:rewrite1} converges to 
\begin{subequations}\label{e:limit}
\begin{align}
\tilde{f}_t(x,t) &= c (\tilde{f}\tilde{g})_x(x,t) +  \frac{\sigma^2}{2}\tilde{f}_{xx}(x,t) \\
\tilde{g}_t(x,t) &= -c (\tilde{f}\tilde{g})_x(x,t) +  \frac{\sigma^2}{2}\tilde{g}_{xx}(x,t),
\end{align}
\end{subequations}
with solutions $\tilde{f} = \tilde{f}(x,t)$ and $\tilde{g} = \tilde{g}(x,t)$ as $k\rightarrow \infty$, $a \rightarrow 0$ 
$ka = c$.

\noindent In this note we analyze the behavior of \eqref{e:limit} as $c \rightarrow \infty$ and illustrate the results with
numerical simulations. The note is organized as follows: in Section \ref{s:structure} we discuss the general 
structure of the BPF model. We identify  the limiting solutions of the Boltzmann price formation model \eqref{e:limit} 
in Section \ref{s:asymptotic}. Finally we illustrate 
the asymptotic behavior of solutions with numerical simulations in Section \ref{s:numerics}.

\section{Structure of the Model}\label{s:structure}

\noindent We start by highlighting some structural aspects of \eqref{e:boltzprice}, which also clarify certain steps in the previous analysis in \cite{BCMW2013,CMP2011,CMW2011}. The general understanding of the structure will serve as a basis for future generalizations and modifications and shall be used in the analysis of the asymptotic case later on. W.l.o.g. we set $\sigma = 1$ throughout this paper.\\

\noindent Let $L$ denote the differential operator $L\varphi = - \varphi_{xx}$, and $S$ and $T$ the shift-operators
$$ (S \varphi)(x)= \varphi(x+a) \qquad (T \varphi)(x)= \varphi(x-a) $$ respectively. 
Then system \eqref{e:boltzprice} becomes
\begin{subequations}
\begin{align}
	f_t + L f &= k ( S - I) (fg), \\
	g_t + L g &= k ( T - I) (fg). 
\end{align}
\end{subequations}
In the setting of kinetic equations $S$ and $T$ are to be interpreted as the gain terms in the collision operators.

\noindent A key property, which allows to derive heat equations for transformed variables, is that $L$ commutes with the collision operators. Hence by defining the formal Neumann series
\begin{align}
	F &:= (I-S)^{-1} f = \sum_{j=0}^\infty S^j f \quad \text{ and } \quad G := (I-T)^{-1} g = \sum_{j=0}^\infty T^j g,
\end{align}
we find that
\begin{subequations}
\begin{align}
	F_t + L F &= -k fg \\
	G_t + L G &= -k fg 	.	
\end{align}
\end{subequations}
Then $F-G$ solves the heat equation. Note that this transformation was already used for the L\&L model \eqref{e:lasrylions} in \cite{CMP2011,CMW2011}
and serve as a key feature of the performed analysis. There the authors motivated the transformation by the structure of the Dirac-$\delta$ terms rather than by inverting the collision operator. Note also that the computations above are purely formal. Since $S$ and $T$ have norm equal to one, the convergence of the Neumann series is not automatically guaranteed and needs to be verified, see \cite{BCMW2013}.\\
\noindent Moreover, also
\begin{subequations}
\begin{align}
	h &= f - (I-S) G \\
	p &= g - (I-T) F, 
\end{align}
\end{subequations}
solve the heat equation. This transformation was used, again without the above interpretation in \cite{BCMW2013}. 

In the special case of the operators above, we have $T=S^{-1}$ and in the $L^2$ scalar product even $T=S^*$, i.e. $S$ and $T$ are unitary operators. Then
$$ (I-S)(I-T)^{-1} = (I-S) \sum_{j=0}^\infty S^{-1} = -S,$$
i.e. we simply have $h = f+Sg$. Note that this structure was exploited in case of the Lasry \& Lions model \eqref{e:lasrylions} in \cite{CMP2011, CMW2011} to derive a-priori estimates.
%

%

\section{Asymptotic behavior when trading with high frequencies}\label{s:asymptotic}

\noindent In this Section we study the limiting behavior of system \eqref{e:limit} as $c \rightarrow \infty$. The limiting
analysis is done in two steps: first by considering the special equilibrated state of system \eqref{e:limit} and then
the full system.\\
\noindent Throughout this paper we make the following assumptions. Let the initial datum $f_I$ and $g_I$ satisfy:
\begin{enumerate}[label=(\textit{\Alph*}), start=1]
\item $f_I,~ g_I \geq 0$ on $\Omega$ and $f_I,~g_I \in \mathcal{S}(\Omega)$,\label{a:schwartz} 
\end{enumerate}
\noindent Let $c = \frac{1}{\varepsilon}$, then system \eqref{e:limit} reads (omitting the tilde)
\begin{subequations}\label{e:limitsys}
\begin{align}
f_t(x,t) &= f_{xx}(x,t) + \prefac (fg)_x\\
g_t(x,t) &= g_{xx}(x,t) + \prefac (fg)_x.
\end{align}
\end{subequations}
Next we reformulate \eqref{e:limitsys} for the new variables $h = f+g$ and $u = f-g$, i.e.
\begin{subequations}\label{e:hu}
\begin{align}
h_t(x,t) - h_{xx}(x,t) &= 0 \label{e:hu_h}.\\
u_t(x,t) - u_{xx}(x,t)  &= \frac{1}{2\varepsilon}(h^2-u^2)_x. 
\end{align}
System \eqref{e:hu} can be considered either on the whole line $\Omega = \R$ or a bounded domain $\Omega = (-1,1)$. Note
that the bounded interval $\Omega$ corresponds to the shifted and scaled interval $(0,p_{\max})$, where $p_{\max}$ denotes the 
maximum price. In the later case system \eqref{e:hu} is supplemented with no flux boundary conditions of the form
\begin{align}\label{e:noflux}
h_x = 0 \text{ and } -u_x = \frac{1}{2\varepsilon}(h^2-u^2) \text{ at } x = \pm 1, 
\end{align}
\end{subequations}
which are equivalent to no-flux boundary conditions for \eqref{e:limitsys}. Throughout this note we consider system \eqref{e:limitsys} on the bounded domain with no-flux boundary conditions \eqref{e:noflux}
only.

\begin{prop}
Let $\varepsilon > 0$ and $f_I$ and $g_I$ satisfy \ref{a:schwartz} and $\Omega = (-1,1)$.
Then system \eqref{e:hu} has a unique smooth solution $(h,u) \in L^\infty(0,T; L^\infty(\Omega))^2 $. Furthermore $u(x,t)^2 \leq h(x,t)^2$ for all $(x,t) \in \Omega \times [0,T]$.\\
\end{prop}

\noindent Note that the functions $f$ and $g$ solve transport diffusion equations, which preserve non-negativity. Trivially 
\begin{align*}
-(f+g) \leq f-g \leq f+g,
\end{align*}
and therefore the inequality $u^2(x,t) \leq h(x,t)^2$ holds.

\subsection{Special case $h=1$}
\noindent We consider the special case $h(x,t) = 1$ as a first step towards understanding the asymptotic
behavior of \eqref{e:limitsys}. Hence it corresponds to the equilibrated solution of the heat equation
\eqref{e:hu_h} on the bounded domain $\Omega = [-1,1]$ with no flux boundary conditions and appropriately
chosen initial datum. Then system \eqref{e:hu} reduces to the viscous Burgers' equation
\begin{subequations}\label{e:burger}
\begin{align}
u_t(x,t) &=u_{xx}(x,t) -\prefac u(x,t) u_x(x,t) \label{e:burgeru}\\
 u(x,0) &= u_I(x) := f_I(x)-g_I(x).
\end{align}
\end{subequations}
The analytic behavior of the classical viscous Burgers' equation (with viscosity $\mu$) for small viscosity in the long-time limit
 was studied in \cite{S1994, KT2001}. The authors showed that a reversal of the limiting passages $t \rightarrow \infty$ and 
$\mu \rightarrow 0$ gives different limiting profiles. Note however that the time scaling of \eqref{e:burger} is different. Equation \eqref{e:burgeru} is a viscous Burgers' equation on a short time scale, a case not considered in the literature so far.

\paragraph{Monotonicity behavior and a-priori estimates of the solution $u$:}
Next we discuss monotonicity properties and a-priori estimates for the solution $u$, which shall be used in the identification of the limiting case $\varepsilon \rightarrow 0$.
\begin{lem}\label{l:monotonicity}
Let $\varepsilon > 0$, $\Omega = (-1,1)$ and let the initial datum $u_I \in \mathcal{S}(\Omega) $. Then the solution $u=u(x,t)$ of \eqref{e:burger} satisfies $u_x(x,t) \leq \max (0,\bar{c})$.
\end{lem}
\begin{proof}
We introduce the function $v = u_x$, which solves
\begin{align}\label{e:v}
v_t(x,t)-v_{xx}(x,t) = -\frac{1}{\varepsilon} (uv)_x(x,t) = -\frac{1}{\varepsilon}(v^2(x,t)+u(x,t)v_x(x,t)),
\end{align}
with $v \leq 0$ at $x=\pm 1$. Then the standard maximum principle implies that $v$ does not attain a positive
maximum inside the parabolic domain. Furthermore the solution $v$ depends continuously on the data, which yields the
desired estimate.
\end{proof}

\noindent Let us consider \eqref{e:burger} on the bounded domain $\Omega$ with no-flux boundary conditions. Then the following a-priori 
estimate for the first order moment holds:
\begin{align*}
\frac{d}{dt} \int_{\Omega} u(x,t) x ~dx &= \int_{\Omega} x (u_{xx}(x,t)+\prefactwo (1-u^2(x,t))_x)~ dx\\
&= - \left[\int_{\Omega} (u_x(x,t) + \prefactwo (1-u^2(x,t)) )~ dx\right].
\end{align*}
Therefore we conclude
\begin{align*}
  \int_0^T \int_{\Omega} (1-u^2(x,t))~dxdt &= 2 \varepsilon \left[ \int_{\Omega} u_I(x) x ~dx - \int_{\Omega} u(x,T) x~dx + \int_0^T \int_{\Omega} u_x(x,t)~dx dt\right]\\
&\leq 2 \varepsilon [ 2 + T (-u(1,t) + u(-1,t))] \leq  4 \varepsilon (1+T),
\end{align*}
using that $\lvert u(x,t) \rvert  \leq h(x,t) = 1$.
In the limit $\varepsilon \rightarrow 0$ we obtain that
\begin{align*}
\int_0^T \int_{\Omega} (1-u^2)~dx dt \rightarrow 0 \text{ as } \varepsilon \downarrow 0.
\end{align*}
Since $u^2 \leq h^2 \leq 1$, this implies that 
\begin{align*}
\int_0^T \int_{\Omega} (s-u)^2~dx dt \leq \int_0^T \int_{\Omega} (1-u^2)~dx dt \leq 4 \varepsilon (1+T),
\end{align*}
for $s = \sgn(u)$. From these estimates we deduce 
\begin{align}\label{e:con}
u^2-1 \rightarrow 0 \text{ in } L^1(\Omega \times (0,T)) \text{ and } u-s \rightarrow 0 \text{ in } L^2 (\Omega\times (0,T)),
\end{align}
for $\varepsilon \rightarrow 0$.
\paragraph{Identification of the limiting function $u$ for $\varepsilon \rightarrow 0$:\\} 
\noindent Finally estimate \eqref{e:con} allows us to identify the limiting functions.

\begin{thm}\label{t:idlimit1}
Let assumption \ref{a:schwartz} be satisfied. Let $m_f = \int_\Omega f~dx $ and $m_g = \int_{\Omega} g~dx$ the mass of buyers and vendors. Then there exists a unique limit $u = u(x,t)$ of the solutions of equation \eqref{e:burger} as $\varepsilon \rightarrow 0$. The limit is given by
\begin{align}\label{e:limitu}
u(x,t) &= 
\begin{cases}
1 & \text{ for } x < \frac{m_f+m_g}{2}\\
-1 & \text{ for } x > \frac{m_f + m_g}{2}.
\end{cases}
\end{align}
\end{thm}
\begin{proof}
First we observe that the total mass of $u$ and $h$ is conserved in time. Since the initial functions $f_I$ and $g_I$ satisfy 
assumption \ref{a:schwartz}, there exists a constant $\bar{c} \in \mathbb{R}$, such that $(u_I)_x(x) \leq \bar{c}$. Hence
$u_x(x,t) \leq \bar{c}$ for all $t > 0$ and the limiting function cannot jump up from $-1$ to $1$. \\
Using estimate \eqref{e:con} we conclude that the limiting function can only take the values $\pm 1$ and has a single
jump down from $1$ to $-1$ at $\tilde{p} \in \Omega$. The location of the jump $\tilde{p}$  (which corresponds to the stationary
price of the traded good) is determined by the conservation of mass, i.e.
\begin{align*}
\int_{-1}^{\tilde{p}} 1~dx - \int_{\tilde{p}}^1 1~dx = m_f - m_g,
\end{align*}
which gives us the limit \eqref{e:limitu}.
\end{proof}

\subsection{Limiting behavior for general $h$:} Next we identify the limiting solutions for the full system \eqref{e:hu},
using the same arguments as in the previous subsection.  
\begin{lem}\label{l:monoton2}
Let $\varepsilon > 0$, $\Omega = (-1,1)$ and let the initial datum $u_I(x)$ satisfy assumption \ref{a:schwartz}.
Then
\begin{align*}
\lvert u \rvert \leq \lvert h \rvert = h \text{ and } 
\int_0^T \int_{\Omega} (h^2-u^2)~dx dt \leq 4 \varepsilon(1+T) \max_{x\in\Omega} h(x,t) .
\end{align*}
\end{lem}
The proof follows the arguments of the previous subsection. From Lemma \ref{l:monoton2} we conclude that 
\begin{align}
u^2 \rightarrow h^2 \text{ in } L^1(\Omega \times (0,T)).
\end{align}

Next we show that the function $u_x$ can not have a jump up from $-h$ to $h$. To do so we consider the function $v = u_x$, which satisfies
\begin{align}
\begin{split}
  v_t(x,t) - v_{xx}(x,t) &= \prefac (h^2(x,t)-u^2(x,t)){}\\ 
&= \frac{1}{\varepsilon} \left[(h(x,t) h_x(x,t))_x - v^2(x,t)\right] - \frac{1}{\varepsilon} u(x,t) v_x(x,t). 
\end{split}
\end{align}
Then the standard maximum principle implies that $v^2(x,t) > \sup_{x\in\Omega}(h(x,t)h_x(x,t))_x$ and
\begin{align*}
u_x(x,t) \leq \max_{x\in\Omega} (\sup_{x\in\Omega}(u_x(x,0)), \sqrt{\max_{x\in\Omega}(0,\sup((hh_x)_x))}).
\end{align*}
Therefore $u = u(x,t)$ cannot have a jump upward and we deduce that the limiting function can be written as
\begin{align}
u(x,t) &= 
\begin{cases}
h(x,t) &\text{ for } x < p(t)\\
-h(x,t) &\text{ for } x > p(t),
\end{cases}
\end{align}
where $p = p(t)$ denotes the position of the jump, i.e. the price of the traded good. It is uniquely determined
for all $t>0$ by
\begin{align}\label{e:masscon}
m_f = \int_{-1}^{p(t)} h(x,t)~dx ~ \text{ or, equivalently } ~ m_g = \int_{p(t)}^1 h(x,t)~dx.
\end{align}
The previous calculations lead to the following theorem:
\begin{thm}
Let assumption \ref{a:schwartz} be satisfied. Then there exist unique limiting functions $(u,h)$ of system \eqref{e:hu}
as $\varepsilon \rightarrow 0$, which are given by
\begin{align*}
u(x,t) &= 
\begin{cases}
h(x,t) \text{ for } x < p(t)\\
-h(x,t) \text{ for } x > p(t),
\end{cases}
\end{align*}
where $p = p(t)$ is determined by \eqref{e:masscon} and $h = h(x,t)$ is the solution of the heat equation \eqref{e:hu_h} with
homogeneous Neumann boundary conditions.
\end{thm}

\begin{rem}\label{r:longtime}
The behavior of the price $p = p(t)$ is determined by the conservation of mass. This implies that
\begin{align*}
m_f-m_g = \int_{\Omega} u(x,t)~dx = \int_{-1}^{p(t)} h(x,t)~dx - \int_{p(t)}^1 h(x,t)~dx.
\end{align*}
Differentiation of the later with respect to time yields $0 = 2 h(p(t),t) p'(t) + 2h_x(p(t),t))$.
Hence we deduce that the evolution of the price in time is given by
\begin{align}
p'(t) = -\frac{h_x(p(t),t)}{h(p(t),t)}.
\end{align}
The function $h = h(x,t)$ solves the heat equation and converges exponentially fast to its steady state, given by
\begin{align*}
h(x,t) \rightarrow \int_{-1}^1 h_I(x)~dx = m_f + m_g \text{ as } t \rightarrow \infty.
\end{align*}
This implies exponential convergence of the price $p = p(t)$, since $p'(t) = - (\ln h(p(t),t))_x$. 
\end{rem}

\section{Numerical simulations}\label{s:numerics}
In this last section we illustrate the behavior of the limiting system with numerical experiments.
All simulations are performed on the interval $\Omega = [-1,1]$ with no-flux boundary conditions
\eqref{e:noflux}. We split the interval into $N=4000$ equidistant intervals for size $\Delta x = 5 \times 10^{-3}$.
System \eqref{e:hu} is discretized using a finite difference discretization, i.e.
\begin{subequations}\label{e:hu_fd}
\begin{align}
\dot{h}_i &= \frac{1}{\Delta x^2} (h_{i+1}-2h_i+h_{i-1})\\
\dot{u}_i &= \frac{1}{\Delta x^2} (u_{i+1}-2u_i+u_{i-1}) + \frac{1}{4 \varepsilon \Delta x }( h_{i+1}^2 - u_{i+1}^2-h_{i-1}^2+u_{i-1}^2).
\end{align}
\end{subequations}
The resulting system of ODEs is solved using an explicit 4th-order Runge-Kutta method (implemented
within the GSL library).  \\

We illustrate the behavior of system \eqref{e:hu} for a not well prepared initial data $f_I$ and 
$g_I$, i.e. the function $f$ is split into two groups with $g$ in between. We choose the following
set of parameters
\begin{align*}
\varepsilon = 5\times 10^{-2} \text{ and } \sigma = 0.1.
\end{align*}
 The evolution of both function is illustrated in Figure \ref{f:notwellprep}.
\begin{figure}
\begin{center}
\subfigure[$t=0$]{\includegraphics[angle=270, width=0.3 \textwidth]{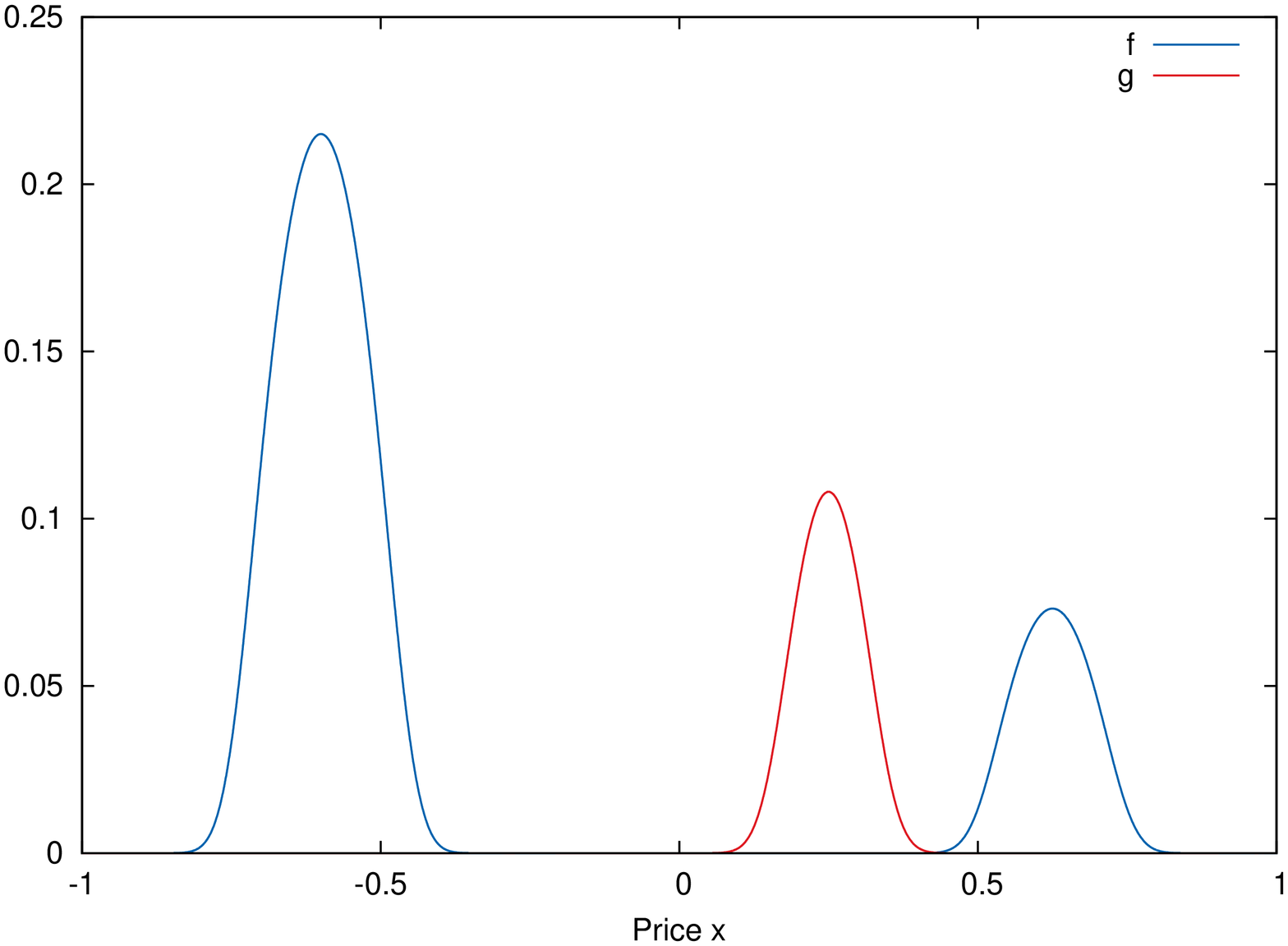}}\hspace*{0.1cm}
\subfigure[$t=0.5$]{\includegraphics[angle=270, width=0.3 \textwidth]{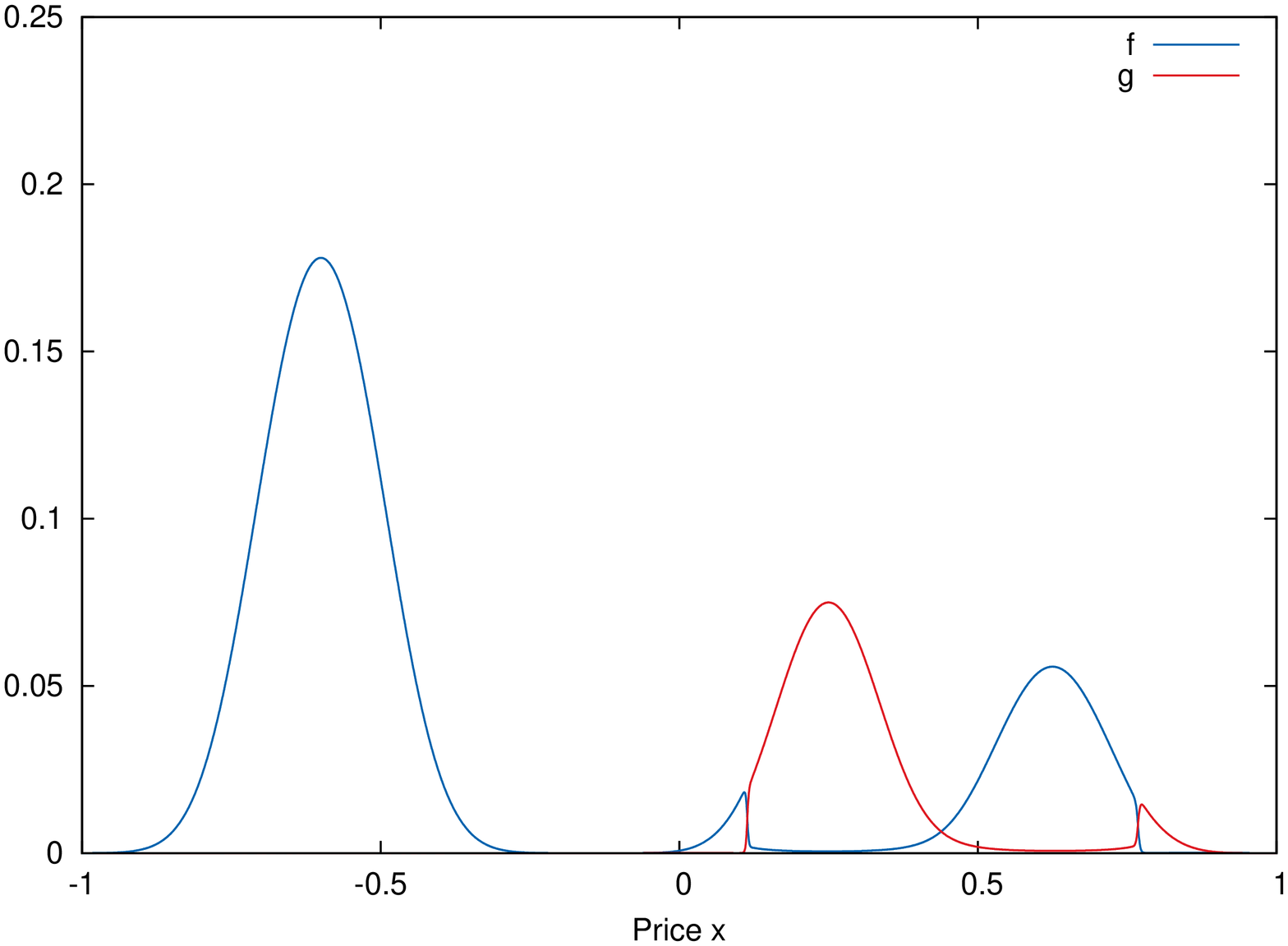}}\hspace*{0.1cm}
\subfigure[$t=1$]{\includegraphics[angle=270, width=0.3 \textwidth]{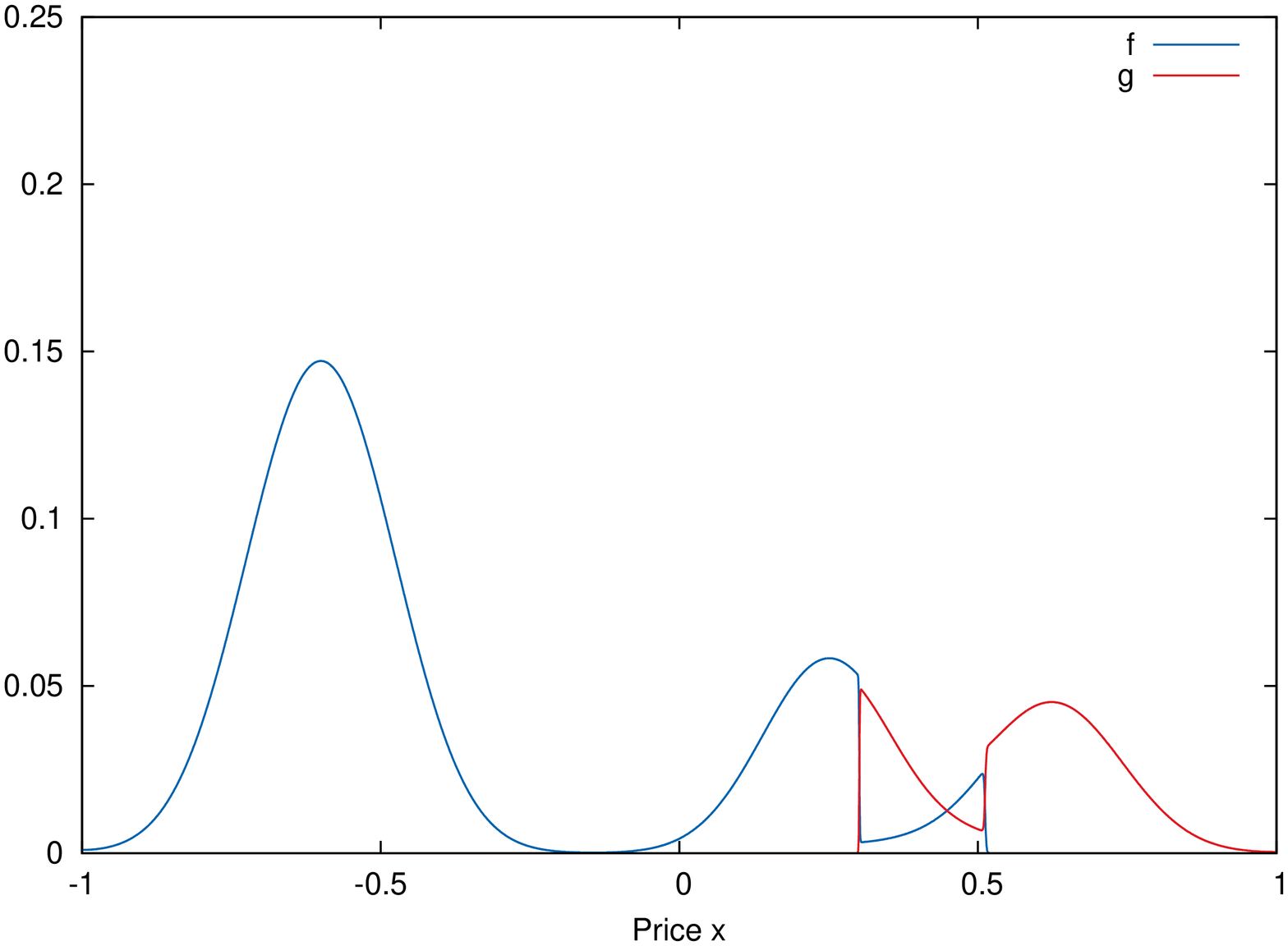}}\\
\subfigure[$t=5$]{\includegraphics[angle=270, width=0.3 \textwidth]{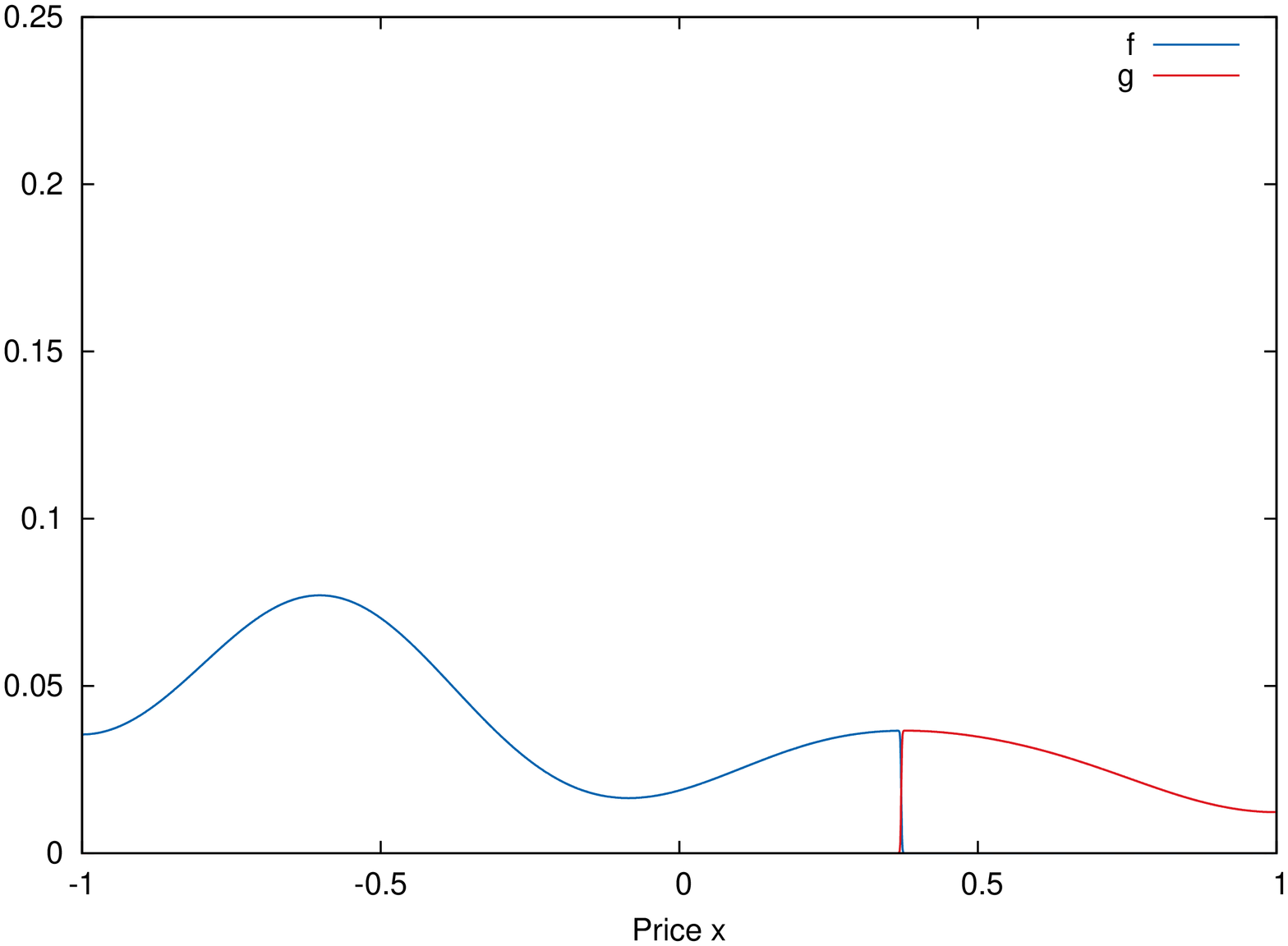}}\hspace*{0.1cm}
\subfigure[$t=10$]{\includegraphics[angle=270, width=0.3 \textwidth]{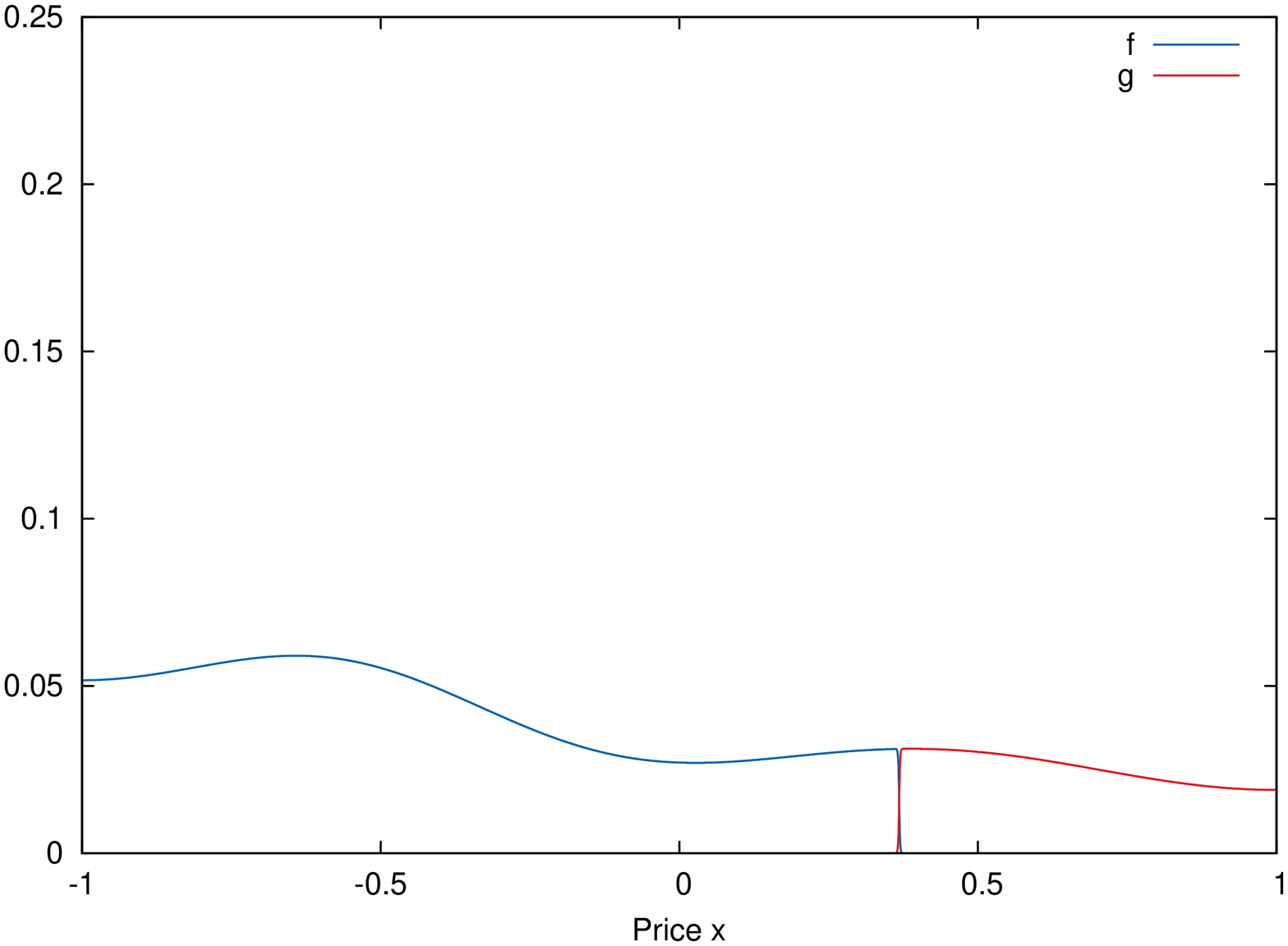}}\hspace*{0.1cm}
\subfigure[$t=30$]{\includegraphics[angle=270, width=0.3 \textwidth]{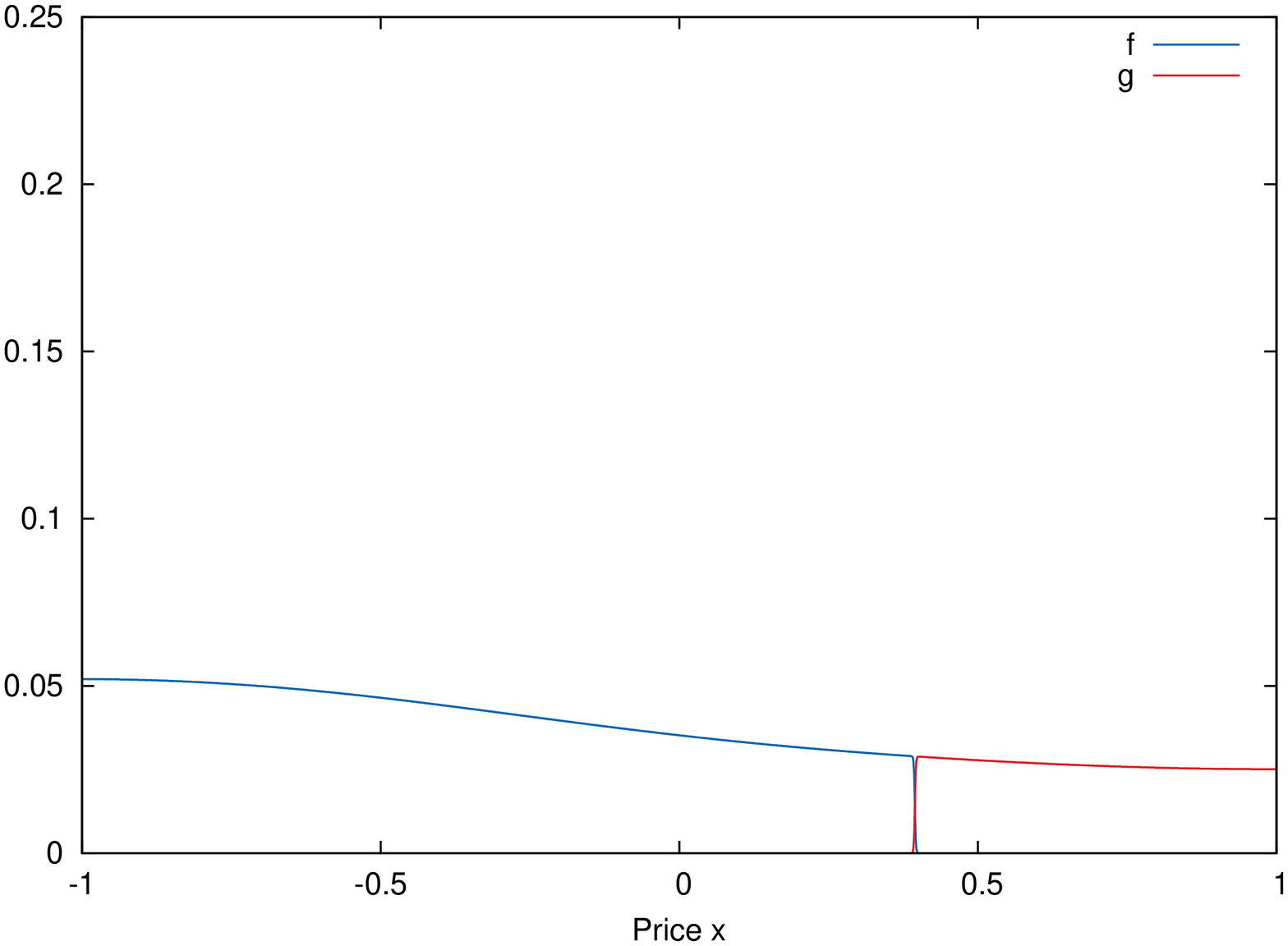}}
\end{center}
\caption{Evolution of the buyer and vendor density in the case of not-well prepared initial data}\label{f:notwellprep}
\end{figure}
We observe the fast segregation of $f$ and $g$ and the formation of a unique interface, which corresponds
to the price $p = p(t)$ in time. This behavior is not unexpected since system \eqref{e:limitsys} has
a similar structure as classical segregation or reaction-diffusion models. Furthermore we observe
a fast equilibration of the price $p = p(t)$ in time, as discussed in Remark \ref{r:longtime}. 

\section{Conclusion}
In this paper we study the asymptotic behavior of a Boltzmann type price formation model, which
describes the trading dynamics in a financial market with high trading frequencies and low transaction
costs. We identify the limiting solutions as the number of transactions tends to infinity and observe
an exponentially fast equilibration of the price in time. Numerical simulations illustrate that 
uneconomic situations, like trading at different prices, are 'corrected' quickly. Hence we conclude
that small fluctuations in the trading frequency or the transaction costs influence the price on
a very short time scale only.
\bibliographystyle{siam}
\bibliography{bpf}

\begin{thebibliography}{10}

\bibitem{BCMW2013}
{\sc M.~Burger, L.~Caffarelli, P.~A. Markowich, and M.-T. Wolfram}, {\em On a
  boltzmann-type price formation model}, Proceedings of the Royal Society A:
  Mathematical, Physical and Engineering Science, 469 (2013).

\bibitem{CMP2011}
{\sc L.~A. Caffarelli, P.~A. Markowich, and J.-F. Pietschmann}, {\em On a price
  formation free boundary model by {L}asry and {L}ions}, C. R. Math. Acad. Sci.
  Paris, 349 (2011), pp.~621--624.

\bibitem{CMW2011}
{\sc L.~A. Caffarelli, P.~A. Markowich, and M.-T. Wolfram}, {\em On a price
  formation free boundary model by {L}asry and {L}ions: the {N}eumann problem},
  C. R. Math. Acad. Sci. Paris, 349 (2011), pp.~841--844.

\bibitem{CGGK2009}
{\sc L.~Chayes, M.~d.~M. Gonz{\'a}lez, M.~P. Gualdani, and I.~Kim}, {\em Global
  existence and uniqueness of solutions to a model of price formation}, SIAM J.
  Math. Anal., 41 (2009), pp.~2107--2135.

\bibitem{GG2011}
{\sc M.~d.~M. Gonz{\'a}lez and M.~P. Gualdani}, {\em Asymptotics for a free
  boundary model in price formation}, Nonlinear Anal., 74 (2011),
  pp.~3269--3294.

\bibitem{KT2001}
{\sc Y.~J. Kim and A.~E. Tzavaras}, {\em Diffusive {$N$}-waves and
  metastability in the {B}urgers equation}, SIAM J. Math. Anal., 33 (2001),
  pp.~607--633 (electronic).

\bibitem{LL2007}
{\sc J.-M. Lasry and P.-L. Lions}, {\em Mean field games}, Jpn. J. Math., 2
  (2007), pp.~229--260.

\bibitem{MMPW2009}
{\sc P.~A. Markowich, N.~Matevosyan, J.-F. Pietschmann, and M.-T. Wolfram},
  {\em On a parabolic free boundary equation modeling price formation}, Math.
  Models Methods Appl. Sci., 19 (2009), pp.~1929--1957.

\bibitem{OHara1998}
{\sc M.~O'Hara}, {\em {Market Microstructure Theory}}, Wiley, Mar. 1998.

\bibitem{S1994}
{\sc J.~Smoller}, {\em Shock waves and reaction-diffusion equations}, vol.~258
  of Grundlehren der Mathematischen Wissenschaften [Fundamental Principles of
  Mathematical Sciences], Springer-Verlag, New York, second~ed., 1994.

\end{thebibliography}

\medskip

{\bf Acknowledgement.}
MTW acknowledges support from the Austrian Science Foundation FWF via the Hertha-Firnberg project T456-N23.
\medskip

\end{document}